\documentclass[11pt]{amsart}

\usepackage{mathrsfs,color}
\usepackage{amsmath,amssymb,amsthm}

\usepackage[pagewise]{lineno}

\input xy
\xyoption{all}
\newtheorem{theorem}{Theorem}[section]

\newtheorem{thm}[theorem]{Theorem}
\newtheorem{cor}[theorem]{Corollary}
\newtheorem{lemm}[theorem]{Lemma}
\newtheorem{prop}[theorem]{Proposition}
\newtheorem{question}[theorem]{Question}

\theoremstyle{definition}

\newtheorem{definition-theorem}[theorem]{Definition-Theorem}
\newtheorem{defi}[theorem]{Definition}
\newtheorem{remk}[theorem]{Remark}
\newtheorem{exam}[theorem]{Example}

\newcommand{\Inv}{\operatorname{Inv}}

\newcommand{\fftors}{\mbox{\rm f-tors}\hspace{.01in}}

\newcommand{\Aa}{\operatorname{\mathcal A}\nolimits}

\newcommand{\TT}{\operatorname{\mathcal T}\nolimits}
\newcommand{\YY}{\operatorname{\mathcal Y}\nolimits}

\newcommand{\Cone}{\mathrm{Cone}}
\newcommand{\FF}{\operatorname{\mathcal F}\nolimits}
\newcommand{\PP}{\operatorname{\mathcal P}\nolimits}

\newcommand{\C}{\mathsf{res}}

\newcommand{\La}{\Lambda}
\newcommand{\stilt}{\mbox{\rm s-tilt}\hspace{.01in}}
\newcommand{\ww}{\mathbf{w}}
\newcommand{\cc}{\mathbf{c}}
\newcommand{\xx}{\mathbf{x}}

\newcommand{\br}[1]{\langle #1 \rangle}
\newcommand{\add}{\mathsf{add}\nolimits}

\newcommand{\id}{\operatorname{id}\nolimits}

\newcommand{\Hom}{\operatorname{Hom}\nolimits}

\newcommand{\Ext}{\operatorname{Ext}\nolimits}
\newcommand{\Tor}{\operatorname{Tor}\nolimits}

\newcommand{\supp}{\operatorname{supp}\nolimits}
\newcommand{\RHom}{\mathbf{R}\strut\kern-.2em\operatorname{Hom}\nolimits}

\def\mod{\mathsf{mod}}

\def\RHom{\mathop{\mathbb R\mathrm{Hom}}\nolimits}

\newcommand{\Fac}{\mathsf{Fac}\hspace{.01in}}

\newcommand{\Sub}{\mathsf{Sub}\hspace{.01in}}
\newcommand{\tors}{\mbox{\rm tors}\hspace{.01in}}
\newcommand{\torf}{\mbox{\rm torf}\hspace{.01in}}

\newcommand{\csort}{c\mbox{\rm -sort}\hspace{.01in}}
\newcommand{\comp}{c\mbox{\rm -antisort}\hspace{.01in}}
\newcommand{\wcsort}{\mbox{\rm b}c\mbox{\rm -sort}\hspace{.01in}}

\newcommand{\quo}{\mbox{\rm cof.quot}\hspace{.01in}}

\def\add{{\mathsf{add}}}

\begin{document}
\title{Torsion pairs for quivers and the Weyl groups}
\author{Yuya Mizuno}
\address{Faculty of Liberal Arts and Sciences, 	Osaka Prefecture University,
  1-1 Gakuen-cho, Naka-ku, Sakai, Osaka 599-8531, Japan}
\email{yuya.mizuno@las.osakafu-u.ac.jp}
\thanks{The first-named author is supported by Grant-in-Aid for JSPS Research Fellow 17J00652.  The second-named author is supported by an NSERC Discovery grant and the
Canada Reseach Chairs program.}

\author{Hugh Thomas}
\address{D\'epartement de math\'ematiques, Universit\'e du Qu\'ebec \`a Montr\'eal, Montr\'eal (Qu\'ebec), H2L 4H2 Canada}
\email{hugh.ross.thomas@gmail.com}
\thanks{}

\begin{abstract}
  We give an interpretation of the map $\pi^c$
  defined by Reading, which is a map from the elements of a Coxeter group
  to the $c$-sortable elements, in terms of the representation theory
  of preprojective algebras.
Moreover, we study a close relationship between $c$-sortable elements and torsion pairs, and give an explicit description of the cofinite torsion classes in the context of the Coxeter group. 
As a consequence, we give a proof of some conjectures proposed by Oppermann,
Reiten, and the second author.
\end{abstract}
\maketitle
\tableofcontents

\maketitle

\section{Introduction}
Preprojective algebras are a fundamental class of algebras, with important
connections to crystal basis theory \cite{KS,BK,BKT}, and which have also been
used to illuminate the structure of significant classes of cluster algebras
\cite{BIRS,GLS,L}.  Recently,
strong links have been discovered between preprojective algebras and
Weyl groups \cite{IR1,BIRS,M}, so that the representation theory of
preprojective algebras can be viewed as providing a categorification of the structure
of the corresponding Weyl group.

More specifically, let $Q$ be a quiver without oriented cycles, let
$\Lambda$ be the corresponding preprojective algebra, and let $W$ be the
corresponding Weyl group.  {(See also {\bf Notation} at the end of this section for further notation.)}
\cite{IR1} introduced a class of ideals
in $\Lambda$ corresponding bijectively to the elements of $W$; we denote the
ideal corresponding to $w\in W$ by $I_w$ (see Theorem \ref{birs}).  Define $\Lambda_w=\Lambda/I_w$.
Then the {objects $I_w$} (respectively, $\Lambda_w$) naturally parametrize a collection of torsion classes (respectively, torsion free classes) in $\mod\La$ as follows. 

\begin{thm}\cite{BIRS}
Let $\tors\La$ $(respectively, \torf\La)$ be the set of torsion classes $($respectively, torsion free classes$)$ in $\mod\La$. 
Then we have injective maps $$W\longrightarrow\tors\La,\ w\mapsto\Fac I_w\ \ \textnormal{and}\ \ W\longrightarrow\torf\La,\ w\mapsto\Sub\La_w,$$
{where $\Fac X$ $(respectively, \Sub X)$ denotes by 
  the category of factor modules $($respectively, submodules$)$ of {direct sums of copies of $X$}.}
\end{thm}


The Weyl group $W$ also has a nice connection to torsion free classes in $\mod KQ$ \cite{AIRT,T}.
Specifically, we have the following result.

\begin{thm}\cite{AIRT,T}\label{intro-airt}
Let $\csort W$ be the set of $c$-sortable elements of $W$.  
Then there is a bijection $$\csort W\longrightarrow\{\textnormal{finite torsion free classes of }\mod KQ\},\ \ \ w\mapsto\C\La_w.$$
\end{thm}
Here, for $X$ a $\Lambda$-module, we
write $\C X$ for the corresponding $KQ$-module defined by restriction of scalars (Definition \ref{restriction}). 
The $c$-sortable elements (Definition \ref{def c-sort}) 
were originally defined by Reading \cite{R2} for the purpose of analyzing
Coxeter group combinatorics associated to cluster algebras and
noncrossing partitions.

Then, there is a natural map from torsion free classes of $\mod\La$ to
torsion free classes of $\mod KQ$, sending $\Sub\La_w$ to the torsion free class $\Sub\La_w\cap\mod KQ$, which turns out to be finite. 
Then by Theorem \ref{intro-airt}, on the level of the Weyl group, this gives us a map from $W$ to the $c$-sortable elements of $W$.  
Our first main result recognizes this map
as one that is
already well-studied in combinatorics.

\begin{thm}\label{intro1}
{We have the following commutative diagram$:$
$$\xymatrix@C20pt@R5pt{
W\ar[rrr]^{\pi^c(-)} \ar[ddd]_{\Sub\La_{-}} &&&{\csort W}\ar[ddd]^{ \C\La_{-}}\\
&&&\\
&&&\\
\torf\La\ar[rrr]^{-\cap\mod KQ}&&&\torf KQ}$$
}
%
\end{thm}

The map $\pi^c$ was
originally introduced by Reading \cite{R2} in purely combinatorial terms: $\pi^c(w)$ is the unique maximal $c$-sortable
element below $w$ (where ``maximal'' is understood with respect to the weak
order on $W$).

The map $\pi^c$ is of particular interest for
the following reason \cite{RS2}.  Let $w$ be a $c$-sortable element of $W$ and let 
$(\pi^c)^{-1}(w)$ denote its inverse image. 
If this inverse image has a maximum element, then  
the union of the
Coxeter chambers corresponding to elements $u\in (\pi^c)^{-1}(w)$
yields a cone in the $g$-vector fan for the cluster algebra whose initial
$B$-matrix is encoded by $Q$.  If the inverse image does not have a maximum
element, then the
corresponding union of Coxeter chambers is the intersection of a cone in the
$g$-vector fan with the Tits cone. All maximal-dimensional cones in the
$g$-vector fan which intersect the Tits cone arise in one of these two ways.

On the other hand, the correspondence between $W$ and $\{I_w\}_{w\in W}$ was enriched in
\cite{ORT} as follows. 

\begin{thm}\cite{ORT}
Let $\quo KQ$ be the set of cofinite (i.e.,there are only finitely many indecomposable modules which are not in the category) quotient closed subcategories of $\mod KQ$. 
Then there is a bijection $$W\longrightarrow\quo KQ,\ \ \ w\mapsto\overline{\C I_w},$$
where $\overline{\C I_w}$ is the additive category generated by $\C I_w$ together with all non-preprojective indecomposable $KQ$-modules. 
\end{thm}

Our investigation has one of its primary origins in the following natural questions and the related conjectures posed in \cite[Conjecture 11.1,11.2]{ORT}.

\begin{question}\label{que 2} 
\begin{itemize}
\item[(a)] When is $\overline{\C I_w}$ a torsion class of $\mod KQ$ for $w\in W$ ? 
\item[(b)] When $\overline{\C I_w}$ is a torsion class, how can we relate $w$ to a $c$-sortable element $x$ which provides the corresponding finite torsion free class $\C\La_{x}$ ?
\end{itemize}
\end{question}

To give an answer, we give the following definition.

\begin{defi}
\label{intro well}
\begin{itemize}
\item[(a)]{If $Q$ is a non-Dynkin quiver, then} 
a $c$-sortable element $x$ is called \emph{bounded} 
if there exists a positive integer $N$ such that $x\leq c^N$.
{If $Q$ is a Dynkin quiver, then}, we regard any $c$-sortable element as bounded. 
We denote by $\wcsort W$ the set of bounded $c$-sortable elements. 
\item[(b)]  Let $x$ be a $c$-sortable element. If there exists a maximum element amongst $w\in W$ satisfying $\pi^c(w)=x$, then we denote it by $\widehat{x}^c=\widehat{x}$ and call it \emph{$c$-antisortable}, following the definition from
  \cite{RS1}. We denote by $\comp W$ the set of $c$-antisortable elements of $W$. 
\end{itemize}
\end{defi}

In connection to these questions, it becomes important to know for which
$c$-sortable elements $w$ there is a maximum element in $(\pi^c)^{-1}(w)$.
This question, though purely combinatorial,
has not been addressed previously in the literature.  
Our second main result provides complete answers for these  questions.

\begin{thm}
We have the following {commutative diagram of} bijections:

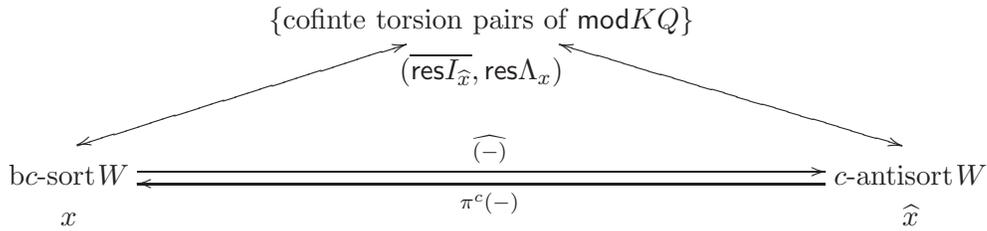
$$\xymatrix@C20pt@R1pt{
& &\{\textnormal{cofinite torsion pairs of }\mod KQ\}\ar@{<->}[rrdddd]\ar@{<->}[lldddd]  &  &   \\
  & &(\overline{\C I_{\widehat{x}}},\C \La_x) & &   \\
&&&&\\
&&&&\\
&&&&\\
\wcsort W\ar@<0.5ex>[rrrr]^{\widehat{(-)}} & & &  &  \comp W\ar@<0.5ex>[llll]^{\pi^c(-)}  \\ 
x& & &  &\widehat{x}
}$$
{In particular, a $c$-sortable element $x$ is bounded if and only if it admits a maximum element in $\{w\in W\ |\ \pi^c(w)=x\}$.}
\end{thm}

Here we call a torsion pair \emph{cofinite} if the torsion class is cofinite. 
Thus we can give a answer to Question \ref{que 2},  confirming a conjecture of \cite[Conjecture 11.1]{ORT}: the cofinite quotient-closed category $\overline{\C I_w}$ is a torsion class precisely if $w$ is $c$-antisortable.
Moreover, in this case the corresponding torsion free class is the one associated to $\pi^c(w)$, confirming a conjecture of \cite[Conjecture 11.2]{ORT}. 
Thus the theorem implies that these torsion pairs can be completely controlled by bounded $c$-sortable elements and $c$-antisortable elements. 
Moreover each category of $(\overline{\C I_{\widehat{x}}},\C \La_x)$ can be described in the context of the Coxeter group and we give explicit correspondences (Theorem \ref{ORT-thm} and \ref{equality}). 
\\

\textbf{Notation.}
Let $K$ be an algebraically closed field. We denote by $D:=\Hom_K(-,K)$. 
Modules mean right modules. For a $K$-algebra $\La$, we denote by 
$\mod\Lambda$ the category of finite dimensional right $\Lambda$-modules. 
Let $X$ be a $\La$-module. 
We denote by $\add X$ (respectively, $\Sub X$, $\Fac X$) the full subcategory whose objects are direct summands (respectively,  submodules, factor modules) of finite direct sums of copies of $X$. \\

\textbf{Acknowledgements.}
The first author would like to thank Osamu Iyama and Takahide Adachi for their valuable comments and stimulating discussions. The second author would like
to thank Steffen Oppermann and Idun Reiten for their collaboration on
\cite{ORT}.  Both authors would also like to express their appreciation to
the Mathematisches Forschungsinstitut Oberwolfach, where our collaboration
began, and to the referee, whose suggestions substantially improved the paper.

\section{Background}

Throughout this paper, let $Q$ be a finite connected acyclic quiver with vertices $Q_0=\{1,\ldots, n\}$. We always assume for simplicity that $Q_0$ are admissibly
numbered, that is, if we have an arrow $j\to i$, then $j<i$.

\subsection{Coxeter groups}\label{back-cg}

\begin{defi}
The \emph{Coxeter group} $W$ associated to $Q$ is defined by the generators $S:=\{s_1,\ldots,s_n\}$ and relations
\begin{itemize}
\item[$\bullet$] $s_i^2=1$,
\item[$\bullet$] $s_is_j=s_js_i$ if there is no arrow between $i$ and $j$ in $Q$,
\item[$\bullet$] $s_is_js_i=s_js_is_j$ if there is precisely one arrow between $i$ and $j$ in $Q$.
\end{itemize}

We denote by $\ww$ a word, that is, an expression in the free monoid generated by $s_i$ for $i\in Q_0$ and $w$ its equivalence class in the Coxeter group $W$. 
Each element $w\in W$ can be written in the form $\ww=s_{u_1}\cdots s_{u_l}$ and,  
if $l$ is minimal among all such expressions for $w$, then $l$ is called the \emph{length} of $w$. In this case, we write $\ell(w)=l$ and we call $s_{u_1}\cdots s_{u_l}$ a \emph{reduced expression}. 
An element $c=s_{u_1}\ldots s_{u_l}$ is called a \emph{Coxeter element} if $l=n$ and $\{u_1,\ldots, u_l\}=\{1,\ldots ,n\}$. 
{A Coxeter element is called \emph{admissible} if $s_j$ precedes $s_i$ in
  any word for $c$ 
  whenever there is an arrow from $j$ to $i$. The assumption that $Q_0$ is admissibly
  numbered is equivalent to the assertion that $s_{1}\ldots s_{n}$ is admissible.}
\end{defi}
We note that our starting point throughout this paper is a quiver $Q$;
the corresponding Cartan matrix is therefore symmetric, which explains
the fact that the braid relations of $W$ are all of the above form.

\begin{defi}\label{def c-sort}
Let $c$ be a Coxeter element. 
Fix a reduced expression of $c$ and regard $c$ as a reduced word.
For $w\in W$, we denote the support of $w$ by $\supp(w)$, that is, the set of generators occurring in a reduced expression of $w$.

We call an element $w\in W$ \emph{c-sortable} if there exists a reduced expression of $w$ of the form $\ww=c^{(0)}c^{(1)}\ldots c^{(m)}$, where all $c^{(t)}$ are subwords of $c$ whose supports satisfy 
$$\supp(c^{(m)})\subset \supp(c^{(m-1)})\subset \ldots \subset \supp(c^{(1)})\subset \supp(c^{(0)})\subset Q_0.$$
\end{defi}

Let $v,w\in W$. We write $v\leq w$ if there exist $s_{u_1},\ldots, s_{u_l}$ such that $w=vs_{u_1}\ldots s_{u_l}$ and $\ell(vs_{u_1}\ldots s_{u_j})=\ell(v)+j$ for $0\leq j\leq l$. 
We call $\leq$ the \emph{$($right$)$ weak order}. 

For the generators $S=\{s_1,\ldots,s_n\}$, we let $\langle s\rangle:=S\setminus \{s\}$ and denote  
$W_{\langle s\rangle }$ by the subgroup of $W$ generated by $\langle s\rangle$.
For any $w\in W$, there is a unique factorization
$w=w_{\langle s\rangle}\cdot {}^{\langle s\rangle}w$ 
maximizing $\ell(w_{\langle s\rangle})$ for $w_{\langle s\rangle}\in W_{\langle s\rangle }$ and $\ell(w_{\langle s\rangle})+\ell({}^{\langle s\rangle}w)=\ell(w)$ 
\cite[section 2.4]{BB}. 

Then we give the following notion introduced by Reading \cite{R2}.

\begin{defi}\label{map pi}
Let $c$ be a Coxeter element and let $s$ be initial in $c$. Then, define 
$\pi^c(\id)=\id$ and, for each $w\in W$, we define 
\[
\pi^c(w):=\left\{\begin{array}{ll}
s\pi^{scs}(sw) &\mbox{if}\ \ell(sw)<\ell(w)\\
\pi^{sc}(w_{\langle s\rangle }) &\mbox{if}\  \ell(sw)>\ell(w).
\end{array}\right.\]
\end{defi}

Then this map gives the following the result.

\begin{thm}\cite[Proposition 3.2]{R3}\cite[Corollary 6.2]{RS2}\label{pai}
For any $w\in W$, 
$\pi^c(w)$ is the unique maximal $c$-sortable element below $w$ in the weak order.
\end{thm}

For the rest of the paper, we fix $c=s_1\dots s_n$ to be an admissible Coxeter
element for $Q$.
One of the aims of this paper is to give a representation-theoretic interpretation of the map $\pi^c$.


\subsection{Preprojective algebras}

\begin{defi}\label{Preprojective algebras}We denote by $Q_1$ the set of arrows of a quiver $Q$. 
The preprojective algebra associated to $Q$ is the algebra
$$\La=K{Q^d}/\langle \sum_{a\in Q_1} (aa^*-a^*a)\rangle$$
where ${Q^d}$ is the double quiver of $Q$, which is obtained from $Q$ by adding for each arrow 
$a:i\rightarrow j$ in $Q_1$ an arrow $a^*:i\leftarrow j$ pointing in the opposite direction. 
\end{defi} 

We now recall an important relationship between preprojective algebras and the Coxeter groups  from \cite{IR1,BIRS}.

Let $\La$ the preprojective algebra of $Q$. 
We denote by $I_i$ the two-sided ideal of $\La$ generated by $1-e_i$, where $e_i$ is a primitive idempotent of $\La$ for $i\in Q_0$.  
We denote by $\langle I_1,\ldots,I_n\rangle$ the set of ideals of $\La$ which can be written as 
$I_{u_l}\cdots I_{u_1}$ for some $l\geq0$ and $u_1,\ldots,u_l\in Q_0$. Then we have the following result (see also \cite[Theorem 2.14]{M} in the Dynkin case).

\begin{theorem}\cite[Theorem III.1.9]{BIRS}\label{birs}
There exists a bijection $W\to\langle I_1,\ldots,I_n\rangle$. 
It is given by $w\mapsto I_w =I_{u_l}\cdots I_{u_1}$ for any reduced 
expression $w=s_{u_1}\cdots s_{u_l}$.
\end{theorem}

Note that the product of ideals is taken in the opposite order to the product of expression of $w$. This follows the convention of \cite{ORT,AIRT}. 


Next we briefly recall the main results of \cite{ORT}, which give a connection between path algebras, preprojective algebras and Coxeter groups.

\begin{defi}\label{restriction}
Let $\La$ be the preprojective algebra of $Q$. 
For a $\La$-module $X$, 
we denote by $X_{KQ}$ the $KQ$-module defined by restriction. 
Moreover we associate the subcategory 
$$\C X=\add X_{KQ} \cap \mod KQ.$$
We denote by $\overline{\C X}$ the additive category generated by $\C X$ together with all non-preprojective indecomposable $KQ$-modules. 
\end{defi}

Consider the infinite word $\cc^\infty:=\cc \cc 
\cc\ldots,$ where $\cc=s_1\ldots s_n$. 
For $w$, 
we take the lexicographically first reduced expression for $w$ in $\cc^\infty$ 
(more explicitly, among all the reduced 
expressions $s_{u_1}\dots s_{u_l}$ for $w$ in $\cc^\infty$, 
we choose the one such that $s_{u_1}$ is as far to the left as possible in $\cc^{\infty}$, and, among such expressions, $s_{u_2}$ is
as far to the left as possible, and so on
for each $s_{u_j}$).
It is uniquely determined and we denote it by $\underline{\ww}$.  
We call it the 
leftmost expression for $w$.  
By identifying $\cc^\infty$ with the indecomposable preprojective $KQ$-modules 
$P_1,\ldots,P_n,\tau^{-1}P_1,\ldots,\tau^{-1}P_n,\tau^{-2}P_1,\ldots$, we have the following result.

\begin{thm}\label{ORT-thm}\cite{ORT} 
The map $w\mapsto \overline{\C I_w}$ gives a bijection between the elements of $W$ and the cofinite (additive) quotient closed subcategories
of $\mod KQ$. 
Moreover, $\overline{\C I_w}$ is given by removing from $\mod KQ$ the indecomposable preprojective $KQ$-modules corresponding to the leftmost word $\underline{\ww}$. 
\end{thm}

Here, a subcategory $\Aa$ in $\mod KQ$ is called \emph{cofinite} if 
there are only finitely many indecomposable $KQ$-modules which are not in $\Aa$. Note that any cofinite quotient closed subcategory contains all the non-preprojective $KQ$-modules \cite[Proposition 2.2]{ORT}.

\begin{exam}
Let $Q$ be the following quiver 
$$\xymatrix@C10pt@R5pt{&2\ar[rd]&\\
1\ar[ru]\ar[rr]&&3.}$$
Let $w=s_1s_3s_2s_3s_1$. 
Then the leftmost word $\underline{\ww}$ is $s_1s_2s_3s_2s_1$. 
Hence the corresponding indecomposable modules are $\{P_1,P_2,P_3,\tau^{-1}P_2,\tau^{-2}P_1\}$ and hence $\overline{\C I_w}$ 
consists of the additive hull of all indecomposable $KQ$ modules other than
these five. 
\end{exam}

Quotient closed subcategories which are also closed under extensions are
called torsion classes.  They are particularly important because of their
connection to tilting theory.
Therefore, it is very natural to ask when $\overline{\C I_w}$ is a torsion class.
The aim of this paper is to give an answer to this question, confirming  \cite[Conjecture 11.1]{ORT}.

\subsection{Support tilting modules}

Next we recall the notion of support tilting modules.

\begin{defi}\cite{IT}
  For a $KQ$-module $X$, we say that $X$ is \emph{tilting} if $\Ext^1(X,X)=0$
  and $X$ has $n$ pairwise non-isomorphic summands, where $n$ is the number
  of vertices of $Q$.  We call $X$ \emph{support tilting} if there exists an idempotent $e$ of $KQ$ such that $X$ is a tilting $(KQ/\langle e\rangle)$-module.
\end{defi}

Then we have the following result (see also \cite[Theorem 2.7]{AIR} for
a more general version of this result).

\begin{thm}\cite[Theorem 2.11]{IT}\label{bij support tilt}
Let $Q$ be an acyclic quiver. 
There is a bijection between the set $\stilt KQ$ of isomorphism classes of basic support tilting $KQ$-modules and the set $\fftors KQ$ of functorially finite torsion classes of $\mod KQ$. It is given by 
$\stilt KQ\ni T\mapsto\Fac T\in\fftors KQ$ and $\fftors KQ\ni\TT\mapsto P(\TT)\in\stilt KQ$, where $P(\TT)$ denotes the direct sum of one copy of each of the
indecomposable Ext-projectives of $\TT$ up to isomorphism.
\end{thm}


\subsection{Sortable elements and finite torsion free classes.}

In this subsection, we review an important connection between $c$-sortable elements and finite torsion free classes.

First we recall layers following \cite{AIRT}. 
For any reduced word $\ww=s_{u_1}\ldots s_{u_l}$, we have the chain of ideals $$\Lambda\supset I_{u_1}\supset I_{u_2}I_{u_1}\supset \ldots \supset I_{u_l}\ldots I_{u_2}I_{u_1}=I_\ww.$$  
For $j=1,\ldots,l$, we define the \emph{layer}  $$L^j_\ww=e_{u_j}L^j_\ww:=\frac{I_{u_{j-1}}\ldots I_{u_1}}{I_{u_j}\ldots I_{u_1}}.$$
Note that any layer $L^j_\ww$ is an indecomposable $\La$-module for any $j=1,\ldots,l$ \cite[Theorem 2.3]{AIRT}.

Then, for a $c$-sortable word, we can give a support tilting $KQ$-module and the associated torsion free class, which can be explicitly described by layers, as follows.

\begin{thm}\cite[Theorem 3.3, 3.11 and Corollary 3.10]{AIRT}\label{equality}
Let $\ww=c^{(0)}c^{(1)}\ldots c^{(m)}=s_{u_1}\ldots s_{u_l}$ be a $c$-sortable word. 
\begin{itemize}
\item[(a)] $L_\ww^j$ is a non-zero indecomposable $KQ$-module for all $j=1,\ldots,l$.
\end{itemize}

Moreover, we denote by $Q^{(0)}$ the quiver $Q$ restricted to the support of $c^{(0)}$. For $i\in Q_0^{(0)}$, we denote by $t_\ww(i)$ the maximal integer such that $u_{t_\ww(i)}=i$ and let 
$$T_\ww:=\bigoplus_{i\in Q^{(0)}_0}L^{t_\ww(i)}_\ww.$$
\begin{itemize}
\item[(b)]$T_{\ww}$ is a tilting $kQ^{(0)}$-module, that is, support tilting $KQ$-module.
\item[(c)] We have $\Sub T_\ww=\add\{L^1_\ww,\ldots, L^{l}_\ww\}=\C \La_w.$
\end{itemize}\end{thm}

\begin{exam}\label{A_2 tilde}

Let $Q$ be the following quiver 
$$\xymatrix@C10pt@R5pt{&2\ar[rd]&\\
1\ar[ru]\ar[rr]&&3.}$$

Then $s_1s_2s_3$ is a Coxeter element of $Q$. 
Let $\ww=s_1s_2s_3s_1s_2s_1$. 
Then we have 
$$L_\ww^1=\begin{smallmatrix} 1\end{smallmatrix},\ 
L_\ww^2={\begin{smallmatrix}2\\ 1\end{smallmatrix}},\ 
L_\ww^3={\begin{smallmatrix}
 &3&&\\1&&2&\\&&&1\end{smallmatrix}},\ 
L_\ww^4={\begin{smallmatrix}
2&&3&&\\&1&&2&\\&&&&1\end{smallmatrix}},\ 
L_\ww^5={\begin{smallmatrix}
\\&3&&&&&\\1&&2&&3&&\\&&&1&&2&\\&&&&&&1\end{smallmatrix}},\ 
L_\ww^6={\begin{smallmatrix}
 3\\1\end{smallmatrix}}.
 $$
Hence we have $T_\ww={\begin{smallmatrix}
 &3&&\\1&&2&\\&&&1\end{smallmatrix}}\oplus{\begin{smallmatrix}
\\&3&&&&&\\1&&2&&3&&\\&&&1&&2&\\&&&&&&1\end{smallmatrix}}\oplus{\begin{smallmatrix}
 3\\1\end{smallmatrix}}$ and 
$$\Sub T_\ww=\add\{L^1_\ww,\ldots, L^{6}_\ww\}.$$
\end{exam}

We call a torsion free class \emph{finite} if it has finitely
many indecomposable modules. 
Theorem \ref{equality} implies that a $c$-sortable element gives a support tilting module and 
the finite torsion free class associated to it.  
Conversely, any finite torsion free classes of $\mod KQ$ is given by a support tilting module induced by a $c$-sortable element as follows.

\begin{thm}\cite[Theorem 3.16]{AIRT}\label{c implies tilt}
Let $\FF$ be a finite torsion free class. 
Then there exists a unique $c$-sortable word $\ww$ such that $T_\ww$ is a support tilting $KQ$-module and $\FF=\Sub T_\ww$.
\end{thm}

Then, combining with Theorems \ref{equality} and \ref{c implies tilt}, we provide the following correspondence, which is also shown in \cite{T}. 

\begin{cor}\cite[Corollary 3.18]{AIRT}\label{sortable bij}
The map $w\mapsto \C \La_w$ gives a bijection 
$$\{ c \mbox{-sortable\ elements}\}\longleftrightarrow\{\mbox{finite\ torsion free\ classes\ of\ }\mod KQ\}.$$
\end{cor}




\subsection{Torsion pairs of preprojective algebras and path algebras.}\label{subsectionTorsion}
In this subsection, we introduce torsion pairs associated to the Coxeter group.

Let $\La$ be the preprojective algebra associated to $Q$.  
For $w\in W$, define the following subcategories 
\begin{align*}
\TT(I_w) :=\{X\in\mod\La\ |\ \Ext_\La^1(I_w,X)=0\},\ \ \ 
\FF(I_w) := \{X\in\mod\La\ |\ \Hom_\La(I_w,X)=0\}.
\end{align*}

Then, $(\TT(I_w),\FF(I_w))$ is a torsion pair in $\mod\La$.  
A simple explanation is, for example, given in \cite[Proposition 2.7]{SY} in the non-Dynkin case and 
 see \cite[section 5]{BKT} in the Dynkin case. 
Moreover, we recall the following result (we refer to \cite[Theorem 5.10]{BKT}). 

\begin{prop}\label{filt layer}
Let $\ww=s_{u_1}\cdots s_{u_l}$ be a reduced expression of $w$. 
Then $X\in\FF(I_w)$ if and only if $X$ has a filtration by the layers $L^1_\ww,\ldots,L^l_\ww$.
\end{prop}

We also use the following result. 

\begin{thm}\cite[Corollary 3.9]{IR2}\label{subobj}
Let $w\in W$ and $\ww=s_{u_1}\ldots s_{u_l}$ be a reduced expression of $w$. 
The objects in $\Sub\La_w$ are exactly the objects in $\mod\La_w$ which have a
filtration by the layers $L^1_\ww,\ldots,L^l_\ww$.
\end{thm}

Thus, as pointed out in \cite[Examples 5.6]{BKT}, we have 
$\Sub\La_w=\FF(I_w)$.

Assume that $Q$ is a non-Dynkin quiver. In this case, $I_w$ is a tilting module \cite[Theorem III.1.6]{BIRS} and 
hence we have $\TT(I_w)=\Fac I_w$ in $\mod\La$. 
Next assume that $Q$ is a Dynkin quiver.  
By \cite[Proposition 4.2]{M}, 
any torsion class (respectively, torsion free class) is given as $\Fac I_w$ (respectively, $\Sub\La_w$) for some $w\in W$.  Further, the torsion free class corresponding
to $\Fac  I_w$ is clearly $\FF(I_w)=\Sub\La_w$, implying the following proposition:

\begin{prop}\label{I_w torsion}
For any $w\in W$, $(\Fac I_w,\Sub\La_w)$ is a torsion pair of $\mod\La$.
\end{prop}

We remark that this torsion pair is the same as the one used in \cite{L} although the convention used there is different from ours (see \cite[3.2.5]{L}).

As a corollary, we have the following result.

\begin{cor}\label{torsionpath}
Let $w\in W$. 
Then $(\Fac I_w\cap\mod KQ,\Sub\La_w\cap\mod KQ)$ is a torsion pair of $\mod KQ$.
\end{cor}

\begin{proof}
  Clearly, we have $\Hom_{KQ}(X,Y)=0$ for any $X\in\Fac I_w\cap\mod KQ$ and $Y\in\Sub\La_w\cap\mod KQ$.
Next, take $X\in\mod KQ$ and assume that we have 
$\Hom_{KQ}(X,Y)=0$ for any $Y\in\Sub\La_w\cap\mod KQ$.
Then we get $\Hom_{\La}(X,Y')=0$ for any $Y'\in\Sub\La_w$ because
the image of a non-zero map from $X$ to $Y'$ would necessarily
  lie in $\Sub\La_w\cap \mod KQ$.
Therefore we get $X\in\Fac I_w$. 
By the dual argument, we get the conclusion.
\end{proof}


\subsection{Partial orders}

Finally, we recall some relationships of
partial orders between elements of Coxeter groups and torsion pairs. 

\begin{prop}\label{orders}
Let $w,v\in W$. 
The following conditions are equivalent.
\begin{itemize}
\item[(i)]$w\leq v$.
\item[(ii)]$\Fac I_w\supset \Fac I_v$ in $\mod\La$.
\item[(iii)]$\Sub\La_w\subset \Sub \La_v$ in $\mod\La$.
\end{itemize}
\end{prop}

\begin{proof}
The equivalence of (b) and (c) follows from the fact that $(\Fac I_w,\Sub\La_w)$ is a torsion pair in $\mod\La$. 
The equivalence of (a) and (c) follows from \cite[II.1]{BIRS} (see also \cite{IR2}).
\end{proof}

Moreover, we have the following results, which are essentially the same as in \cite[Lemma 10.5]{ORT}, though we do not assume that $Q$ is Dynkin.

\begin{prop}\label{pathorders} 
Let $x,y\in W$ be $c$-sortable elements. 
Then the following conditions are equivalent.
\begin{itemize}
\item[(i)]$x\leq y$.
\item[(ii)]${}^{\perp}(\C \La_x)\supset {}^{\perp}(\C \La_y)$.
\item[(iii)]$\C \La_x\subset \C \La_y$.
\end{itemize}
\end{prop}

\begin{proof}The equivalence of (b) and (c) follows from the fact that $({}^{\perp}(\C \La_x),\C \La_x)$ is a torsion pair in $\mod KQ$. 
Let $\xx=s_{u_1}\ldots s_{u_l}$ be a reduced expression of $x$.
Then, by \cite[Theorem 2.7]{AIRT}, the dimension vectors of the layers 
$L_\xx^j$ ($1\leq j\leq l$) are given by positive (real) roots, 
and the set of positive roots which appear does not depend on the choice of reduced expressions of $x$. 
On the other hand, there exists a unique indecomposable $KQ$-module which has the same dimension vector as $L_\xx^j$ by \cite{K}. Hence, by Theorem \ref{equality}, we have the equivalence of (a) and (c).
\end{proof}


\section{A representation-theoretic interpretation of the map $\pi^c$}
Let $Q$ be a finite connected acyclic quiver with vertices $Q_0=\{1,\ldots, n\}$. We assume  that $Q_0$ are admissibly numbered, that is, if we have an arrow $j\to i$, then $j<i$.  
(Hence 1 is a source of $Q$).
Let $\La$ be the preprojective algebra of $Q$. 

Let $Q':=\mu_1(Q)$ be the quiver obtained by reversing all arrows associated to the vertex 1. 
Then let $T:=\tau^- S_1 \oplus KQ/ S_1$ and denote the reflection functors by  
$R^+:=\Hom_{KQ}(T,-)$ and $R^-:=-\otimes_{KQ'}T$. 
These functors induce quasi-inverse equivalences 
$$\xymatrix@C10pt@R5pt{
\mod KQ/[S_1] \ar@<1mm>[rrr]^{R^+} &&& \mod KQ'/[S'_1],\ar@<1mm>[lll]^{R^-}}$$
where $\mod KQ/[S_1]$ (respectively, $\mod KQ'/[S'_1]$ ) 
is obtained from $\mod KQ$ (respectively, $\mod KQ'$) by annihilating morphisms factoring through the simple projective $KQ$-module $S_1$ (respectively, the simple injective $KQ'$-module $S'_1$).

We denote by $\overline{Q}$ 
the quiver given by removing the vertex 1 and the associated arrows.
Then, the first main result is given as follows (c.f. Definition \ref{map pi} and Theorem \ref{pai}).

\begin{thm}\label{main1}
\begin{itemize}
\item[(a)]Let $w\in W$ which is not the identity.
Then we have 
\[
\Sub\La_w\cap\mod KQ=\left\{\begin{array}{ll}
\add\{R^-(\Sub \La_{s_1w}\cap\mod KQ'), S_1\} &\mbox{if}\ \ell(s_1w)<\ell(w)\\
\Sub\La_{w_{\langle s_1\rangle}}\cap\mod K\overline{Q} &\mbox{if}\  \ell(s_1w)>\ell(w).
\end{array}\right.\]
\item[(b)]Let $w\in W$. Then we have 
$$\Sub\La_w\cap\mod KQ=\C \La_{\pi^c(w)}\ \ \mbox{and}\ \  \Fac I_w\cap\mod KQ=
{}^{\perp}(\C \La_{\pi^c(w)}).$$
\end{itemize}
\end{thm}

{Theorem \ref{intro1} is an immediate consequence.}
We give the proof of Theorem \ref{main1} after developing some lemmas. 

First we recall that the above subcategory $\add\{R^-(\Sub \La_{s_1w}\cap\mod KQ'), S_1\}$ is a torsion free class of $\mod KQ$. 

\begin{lemm}{\cite[Proposition 4.4]{T}}\label{add torsion-free}
Let $\FF'$ be a torsion free class of $\mod KQ'$. 
Then the subcategory  
$$\add \{R^-(\FF'),S_1\}$$ 
 is a torsion free class of $\mod KQ$.
\end{lemm}

This implies the following conclusion.

\begin{prop}\label{torsion sub}
The subcategory $\add\{R^-(\Sub\La_{w}\cap\mod KQ'), S_1\}$ is  a torsion free class of $\mod KQ$.
\end{prop}

\begin{proof} 
$\Sub\La_{w}\cap\mod KQ'$ is a torsion free class by Corollary \ref{torsionpath}. 
Thus, Lemma \ref{add torsion-free} shows the assertion.
\end{proof}

Next we give the following observation.

\begin{lemm}\label{ext closed}
Let $w\in W$ and $\ww=s_{u_1}s_{u_2}\cdots s_{u_l}$ be a reduced expression for $w$.
Let $\mathcal{L}^{Q}_\ww:=\{L_\ww^j\ (1\leq j\leq l)|\ L_\ww^j : KQ\mbox{-module} \}$. 
Then we have
$$\Sub\La_w\cap\mod KQ=\mathsf{Ext}_{}(\mathcal{L}^{Q}_\ww),$$
where the right hand side denotes the smallest extension closed subcategory of $\mod KQ$ containing $\mathcal{L}^{Q}_\ww$.
\end{lemm}

\begin{proof}
By Theorem \ref{subobj}, we have $\mathcal{L}^{Q}_\ww\subset\Sub\La_w\cap\mod KQ$. Since the category is extention closed, we have $\mathsf{Ext}_{}(\mathcal{L}^{Q}_\ww)\subset \Sub\La_w\cap\mod KQ$. 
We will show the converse.
Let $X\in\Sub\La_w\cap\mod KQ$. 
Then $X$ has a filtration by the layers $L_\ww^j\ (1\leq j\leq l)$ from Theorem \ref{subobj}. 
Moreover, since $X$ is a $KQ$-module, these layers consist of $KQ$-modules. 
Therefore we get $X\in\mathsf{Ext}_{}(\mathcal{L}^{Q}_\ww)$.
\end{proof}



Following the notation of \cite{SY}, we let 
\begin{eqnarray*}
\YY(I_1)&:=&\{X\in\mod\La\ |\Tor_1^\La(X,I_1)=0\}\\
&=&\{X\in \mod\La\ |S_1\mbox{ is not a direct summand of the socle of } X\}.
\end{eqnarray*}

(We refer to \cite[Lemma 2.23]{SY} and \cite[Section 5]{BKT} for the above equality.) 
Then we have the following lemma.

\begin{lemm}\label{socle vanish}
Let $w\in W$ and assume that $\ell(s_1w)>\ell(w)$.
Then we have 
$$\Sub\La_w\subset \YY(I_1).$$ 
\end{lemm}

\begin{proof}
Assume that $\Sub\La_w\not\subset\YY(I_1).$
Then $S_1$ is a direct summand of the socle of some module $X\in\Sub\La_w$.
In particular $S_1\in\Sub\La_w$. 

On the other hand, Theorem \ref{subobj} implies that any element in $\Sub\La_w$ is given by a filtration of the layers and hence $S_1$ is one of the layers. 
Moreover, by \cite[Theorem 2.7]{AIRT}, the dimension vectors of the layers are the positive roots  
$s_{u_1}s_{u_2}\cdots s_{u_{j-1}}\alpha_{u_j}$ ($1\leq j\leq l$), where 
$\ww=s_{u_1}s_{u_2}\cdots s_{u_l}$ is a reduced expression for $w$.
If $s_{u_1}s_{u_2}\cdots s_{u_{j-1}}\alpha_{u_j}=\alpha_1$, then
  $s_1s_{u_1}s_{u_2}\cdots s_{u_{j-1}}s_{u_j}=s_{u_1}s_{u_2}\cdots s_{u_{j-1}}$.
Since  $\ell(s_1w)>\ell(w)$, this is impossible.
\end{proof}

Moreover we recall the following results from \cite{AIRT} (see also \cite[Proposition 7.1]{BK}).

\begin{prop}\label{layer and reflection}
Let $\ww=s_1s_{u_2}\cdots s_{u_l}$ be a reduced expression and 
$\ww'=s_{u_2}\cdots s_{u_l}$. 
Then, for $2\leq j\leq l$, we have 
$$L^{j}_{\ww}\simeq L^{j-1}_{\ww'}\otimes_\La I_1.$$

Moreover,\begin{itemize}
\item[(a)]if $L^{j-1}_{\ww'}$ is a $KQ'$-module, 
then we have $L^{j}_{\ww}\simeq L^{j-1}_{\ww'}\otimes_\La I_1\simeq L^{j-1}_{\ww'}\otimes_{KQ'} T.$
In particular, $L^{j}_{\ww}$ is $KQ$-module.
\item[(b)]if $L^{j}_{\ww}$ is $KQ$-module, then we have 
$L^{j-1}_{\ww'}\simeq\Hom_\La(I_1,L^{j}_{\ww})\simeq\Hom_{KQ}(T,L^{j}_{\ww})$.
In particular, $L^{j-1}_{\ww'}$ is a $KQ'$-module. 
\end{itemize}
\end{prop}

\begin{proof}
The first statement follows from \cite[Proposition 2.2]{AIRT}. 

(a) Assume that  $L^{j-1}_{\ww'}$ is a $KQ'$-module. Then \cite[Corollary 2.12]{AIRT} implies that 
$L^{j-1}_{\ww'}\otimes_\La I_1\simeq L^{j-1}_{\ww'}\otimes_{KQ'} T$ and hence it follows from the first statement.

(b) Dually, we obtain $\Hom_\La(I_1,L^{j}_{\ww})\simeq\Hom_{KQ}(T,L^{j}_{\ww})$, which is a $KQ'$-module. 
On the other hand, because we have an equivalence (for example \cite[Lemma 2.16]{SY} and \cite[Section 5]{BKT})
$$\xymatrix@C10pt@R5pt{\TT(I_1) \ar@<1mm>[rrrr]^{\Hom_\La(I_1,-)} &&&& \YY(I_1)\ar@<1mm>[llll]^{-\otimes_\La I_1}},$$ 
we get 
$\Hom_{\La}(I_1,L^{j}_{\ww}) \simeq\Hom_{\La}(I_1,L^{j-1}_{\ww'}\otimes_\La I_1)\simeq L^{j-1}_{\ww'}.$  
This shows the assertion.
\end{proof}

Then we give the following proposition, which implies the first equation of Theorem \ref{main1} of (a).

\begin{prop}\label{sw<w}
If $\ell(s_1w)<\ell(w)$, then we have 
$$\Sub\La_w\cap\mod KQ= \add\{R^-(\Sub\La_{s_1w}\cap\mod KQ'), S_1\} .$$
\end{prop}

\begin{proof}
Let $\ww=s_1s_{u_2}\cdots s_{u_l}$ be a reduced expression of $w$ and 
$\ww'=s_{u_2}\cdots s_{u_l}$. 

First, by Lemma \ref{ext closed}, we have 
\begin{eqnarray}\label{1}\Sub\La_{w}\cap\mod KQ=\mathsf{Ext}(\mathcal{L}^{Q}_{\ww}).\end{eqnarray}

Similarly we have $\Sub\La_{s_1w}\cap\mod KQ'=\mathsf{Ext}(\mathcal{L}^{Q'}_{\ww'})$. Moreover, Lemma \ref{socle vanish} implies this category 
does not contain $S_1'$, the simple injective $KQ$-module at 1. 
Then, because the functor $R^-$ gives an equivalence between $\mod KQ/[S_1]$ and $\mod KQ'/[S'_1]$, we have 
\begin{eqnarray}\label{2}
R^-(\Sub\La_{s_1w}\cap\mod KQ')
=R^-(\mathsf{Ext}(\mathcal{L}^{Q'}_{\ww'}))=\mathsf{Ext}(R^-(\mathcal{L}^{Q'}_{\ww'})).\end{eqnarray}

On the other hand, by Lemma \ref{add torsion-free}, we have 
\begin{eqnarray}\label{3}
\ \ \mathsf{Ext}(R^-(\Sub\La_{s_1w}\cap\mod KQ'),S_1)=\add\{R^-(\Sub\La_{s_1w}\cap\mod KQ'),S_1\}.\end{eqnarray}

Thus, from (\ref{2}) and (\ref{3}),
we get  
\begin{eqnarray}\label{4}\add\{R^-(\Sub\La_{s_1w}\cap\mod KQ'), S_1\}=\mathsf{Ext}(R^-(\mathcal{L}^{Q'}_{\ww'}),S_1).\end{eqnarray} 
Therefore, by (\ref{1}) and (\ref{4}), 
it is enough to show 
$$\mathsf{Ext}(\mathcal{L}^{Q}_{\ww})=\mathsf{Ext}(R^-(\mathcal{L}^{Q'}_{\ww'}),S_1).$$
Because $L^{1}_{\ww}=S_1$, Proposition \ref{layer and reflection} shows that 
$\mathcal{L}^{Q}_{\ww}=\{R^-(\mathcal{L}^{Q'}_{\ww'}),S_1\}$ and the conclusion follows. 
\end{proof}

Next we deal with the case of $\ell(s_1w)>\ell(w)$.

\begin{prop}\label{sw>w}
If $\ell(s_1w)>\ell(w)$, then we have 
$$\Sub\La_w\cap\mod KQ=\Sub\La_{w_{\langle s_1\rangle}}\cap\mod K\overline{Q} .$$
\end{prop}

\begin{proof}
We will show $$\Sub\La_w\cap\mod KQ\subset \Sub\La_{w_{\langle s_1\rangle}}\cap\mod K\overline{Q} $$
and the converse is obvious by Proposition \ref{orders}. 

First we will show $\Sub\La_w\cap\mod KQ\subset\mod K\overline{Q}$. 
Note that we have $\La_c\simeq KQ$ in $\mod\La$ and hence a $KQ$-module is a $\La$-module annihilated by $I_c$ \cite[Lemma 2.11]{AIRT}.

Let $X\in\Sub\La_w\cap\mod KQ$. 
Since $X$ is a $KQ$-module, we have $X\cdot I_c=0$. 
Hence we have 
$X\cdot I_c=(X\cdot I_n\cdots I_2)\cdot I_1=0.$
Thus, we get $Y:=X\cdot (I_n\cdots I_2)\in\add(S_1)$. 
On the other hand, we have $X\in\YY(I_1)$ by Lemma \ref{socle vanish}. Therefore we conclude $Y=X\cdot( I_n\cdots I_2)=0$ and hence $X$ is a $K\overline{Q}$-module.  

Then we will show $\Sub\La_w\cap\mod KQ\subset \Sub\La_{w_{\langle s_1\rangle}}\cap\mod K\overline{Q}$. 
If $w=w_{\langle s_1\rangle}$, then it is clear. 
Let $w=w_{\langle s_1\rangle}{}^{\langle s_1\rangle}w$ and assume that ${}^{\langle s_1\rangle}w$ is not the identity. 
Let $\ww=s_{u_1}s_{u_2}\cdots s_{u_h}s_{u_{h+1}}\cdots s_{u_l}$ be a reduced expression of $w$ such that $\ww_{\langle s_1\rangle}=s_{u_1}s_{u_2}\cdots s_{u_h}$. 

By the above argument and Lemma \ref{ext closed}, it is enough to show that $L_\ww^j$ is not $K\overline{Q}$-module if $j>h$. 
Assume that $L_\ww^j$ is $K\overline{Q}$-module for some $j>h$. 
Then, by \cite[Theorem 2.7]{AIRT}, the dimension vectors of the layer 
$L_\ww^j$ is given by a (real) root of $\overline{Q}$
which contradicts the maximality of $\ell(w_{\langle s_1\rangle})$. 
Thus our claim follows.
\end{proof}

As a consequence, we give a proof of Theorem \ref{main1}.

\begin{proof}[Proof of Theorem \ref{main1}]
(a) follows from Propositions \ref{sw<w} and \ref{sw>w}. 

We will show (b). 
By Corollary \ref{torsionpath}, 
$(\Fac I_w\cap\mod KQ,\Sub\La_w\cap\mod KQ)$ is a torsion pair of $\mod KQ$. Therefore, the second equality follows from the first one. 
We will show the first one.

If $w$ is the identity, then it is clear. 
Assume that $w$ is not the identity.  
We will show the statement by induction on the rank of $Q$ and the length of $w$. 

Let $\mathcal{F}:=\Sub\La_w\cap\mod KQ.$ 
First, assume that $\ell(s_1w)<\ell(w)$. 
Then, by (a), we have 
$$\mathcal{F}=\add\{R^-(\Sub\La_{s_1w}\cap\mod KQ'), S_1\}.$$

Let $c':=s_1cs_1$. 
By induction on the length, we have 
$$\Sub\La_{s_1w}\cap\mod KQ'=\C \La_{\pi^{c'}(s_1w)}.$$ 
Since $s_1\pi^{c'}(s_1w)=\pi^{c}(w)$ is a $c$-sortable element, 
 \cite[Theorem 3.8]{AIRT} implies that  
$$\add\{R^-_1(\C \La_{\pi^{c'}(s_1w)}),S_1\}=\C \La_{\pi^{c}(w)}.$$

Next, assume that $\ell(s_1w)>\ell(w)$. 
Then, by (a), we have 
$$\mathcal{F}=\Sub\La_{w_{\langle s_1\rangle}}\cap\mod K\overline{Q}.$$
By induction on the rank, we have 
$$\Sub\La_{w_{\langle s_1\rangle}}\cap\mod K\overline{Q}=\C \La_{\pi^{s_1c}(w_{\langle s_1\rangle})}.$$
Since we have $\pi^{c}(w)=\pi^{s_1c}(w_{\langle s_1\rangle})$, we obtain the assertion.
\end{proof}


  \section{Combinatorics of Cambrian cones}

In subsequent sections, we need the following combinatorial
results on $\pi^c$. {Recall that if $y$ is a $c$-sortable element,
we write $\widehat y^c$ for the maximum element of the $\pi^c$ fibre over $y$,
if there is one.}

\begin{lemm}\label{lemmatwo} $\{w\in W\ |\ \pi^c(w)=y\}$ is finite if and only if there exists $\widehat{y}^c$. 
\end{lemm}


\begin{prop}\label{inductive}
Suppose $x$ and $y$ are $c$-sortable elements such that 
$y$ covers $x$ in the partial order on $c$-sortable elements.
If there exists $\widehat{y}^{c}$, 
then there exists $\widehat{x}^c$ and $\widehat{y}^c>\widehat{x}^c$. 
\end{prop}

\begin{lemm}\label{max is max} Let $y$ be $c$-sortable.  If
  $\{w \mid \pi^c(w)=y\}$ has a maximal element, then that element
  is actually the maximum.  \end{lemm}

In this section, we establish these results, using
the \emph{Cambrian cones} studied by Reading-Speyer (we refer to \cite{RS1,RS2}, but see also \cite{R1,R2,R3} for earlier work). We begin by briefly
recalling some background.

We recall some basic terminology (see \cite{BB,H,Bo}).
Let $V$ be a real vector space of dimension $n$ with a basis $\alpha_i$ ($i\in Q_0$) and let $V^*$ be the dual vector space with {dual} basis $\rho_i$. 
We write $\br{x^*,y}$ for the canonical pairing of $x^*\in V^*$ and $y\in V$. 
We define a symmetric bilinear form on $V$ by $(\alpha_i,\alpha_i)=2$ for any $i\in Q_0$ and for $i\neq j$, $(\alpha_i,\alpha_j)$ is the negative of the number of edges between the 
vertices $i$ and $j$ of $Q$. 
We define $s_i(\alpha_j)=\alpha_j-(\alpha_i,\alpha_j)\alpha_i$.  The group
generated by these reflections is the Coxeter group $W$ associated to $Q$. 
Note that $w\in W$ naturally acts on $V$ and $V^*$, 
and $\rho_i$ is fixed by $W_{\langle s_i\rangle}$.

Let $\Phi=\{w(\alpha_i)\}_{i\in Q_0,w\in W}$ be the set of (real) roots and 
$\Phi^+$ the set of positive roots. 
To each root $\beta\in\Phi^+$, define a hyperplane 
$H_\beta:=\{v^*\in V^*\ |\ \langle v^*,\beta\rangle=0\}.$  The connected
components of $V^*\setminus \bigcup_{\beta\in \Phi^+}H_\beta$ we refer to as \emph{chambers}.
Let $D:=\bigcap_{i \in Q_0} \{v^*\in V^*\ |\ \br{v^*,\alpha_i}\ge 0\}$ be the \emph{dominant chamber}. It gives a fundamental domain for the action of $W$ on the Tits cone $\cup_{w\in W} wD$. We refer to $\{wD\mid w\in W\}$ as the collection of
$W$-chambers.  
We say that $wD$ is below $H_\beta$ (respectively, above $H_\beta$) if it is contained in $H_\beta^+:=\{v^*\in V^*\ |\ \langle v^*,\beta\rangle\geq0\}$ (respectively,  $H_\beta^-:=-H_\beta^+$). 
Moreover, let $\Phi_{\langle s_i\rangle}$ denote the roots
  corresponding to reflections of $W_{\langle s_i\rangle}$, and
let $V_{\langle s_i\rangle}$ be the
hyperplane in $V$ spanned by $\Phi_{\langle s_i\rangle}$. Let
$D_{\langle s_i\rangle}:=\bigcap_{ j \in Q_0\setminus\{i\}} \{v^*\in V_{\langle s_i\rangle}^*\ |\ \br{v^*,\alpha_j}\ge 0\}.$ 

We denote by $d(wD,vD)$ the set of hyperplanes which separate $wD$ and $vD$.
Then the following result is well-known \cite{Bo}.

\begin{lemm}\label{partial cone}
Let $w,v\in W$. The following conditions are equivalent.
\begin{itemize}
\item[(i)] $w\geq v$. 
\item[(ii)] $d(wD,D)\supseteq d(vD,D)$.
\end{itemize}
\end{lemm}

{For the remainder of the section}, let $s$ be an initial reflection of $c$ up to commutation.
(We can simply take $s=s_1$.)  We write $\alpha_s$ for the corresponding simple
root, and $H_s$ for the corresponding hyperplane.
Now, we state a definition from \cite{RS2}, which associates to each  $c$-sortable element $x$ of $W$ a set of $n$ vectors, denoted $C_c(x)$. This construction is inductive; the definition of $C_c(x)$ assumes that the map has
  already been defined for standard parabolic subgroups of $W$, and for elements of $W$ whose length is smaller than that of $x$.
For a $c$-sortable element $x$,  we define
\[C_c(x)=\left\lbrace\begin{array}{ll}
C_{sc}(x)\cup\{\alpha_s\}&\mbox{if } \ell(sx)> \ell(x)\\
s C_{scs}(sx)&\mbox{if } \ell(sx)<\ell(x).
\end{array}\right.\]
The basis of
  the induction is the trivial subgroup of $W$, for which we have $C_e(e)=\emptyset$.
Note that $C_c(x)$ is a basis of $V$. 
Then define the \emph{$c$-Cambrian cone}
\[\Cone_c(x)=\bigcap_{\beta \in C_c(x)}\{v^*\in V^*\ |\ \br{v^*,\beta} \geq 0\}.\]

Then the following is one of the main results of \cite{RS2}, which gives a geometric description of the fibers of $\pi^c$.

\begin{theorem} \label{pidown fibers}
  Let $x$ be $c$-sortable.
\begin{itemize}
\item[(a)]  Then $\pi^c(w)=x$ if and only if $w D \subset \Cone_c(x)$ for $w\in W$. 
\item[(b)]  Moreover, the set $\{w\in W\ |\ \pi^c(w)=x\}$ is finite if and only if 
$\Cone_c(x)$ is a finite union of Coxeter chambers.
\end{itemize}
\end{theorem}

Statement (a) of the theorem is 
  \cite[Theorem 6.3]{RS2}. The additional content of (b)
  is that if $\{w\in W\ |\ \pi^c(w)=x\}$ is finite then $\Cone_c(x)$ cannot
  extend outside the Tits cone. If it did, then, since it also contains the
  chamber $xD$ which is inside the Tits cone, it would necessarily also contain in its interior a point on the boundary of the Tits cone. But this would imply
  that it contains an infinite number of $W$-chambers, which would contradict
  (a).

\begin{exam}\label{chamber}
\begin{itemize}
\item[(a)] Let $Q$ be a quiver of type $A_2$. Then $W$-chambers are given as follows.
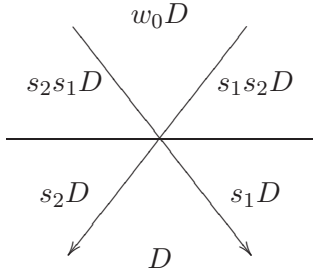
\[\xymatrix@C4pt@R10pt{   & &  w_0D &&\\
& s_2s_1D&  &s_1s_2D&\\
\ar@{-}[rrrr]^{}&  &  &&\\
& s_2D&  &s_1D&\\
&\ar@{<-}[rruuuu]  &  D&\ar@{<-}[lluuuu]& \\
}\]

Then, for example, $C_c(s_2)=\{-\alpha_2,\alpha_1\}$ and $\Cone_c(x)$ is a union of 
$s_2D$ and $s_2s_1D$. These cones correspond to elements such that $\{w\in W\ |\ \pi^c(w)=s_2\}$. 

\item[(b)]Let $Q$ be a quiver of type $\tilde{A}_1$. 
  Then one can check that any element of $s_2$, $s_2s_1$, $s_2s_1s_2$, $\ldots,$ satisfies $\pi^c(w)=s_2$ and these cones are contained in (but do not
  exhaust) $\Cone_c(s_2)$.
\end{itemize}
\end{exam}

Consider the reflection hyperplanes corresponding to the
  reflections in $W_{\langle s\rangle}$. These hyperplanes are perpendicular
  to roots in  the root system $\Phi_{\langle s\rangle}$, and therefore all pass
  through $\rho_s$. Quotienting by $\mathbb R\rho_s$, we would obtain the
  Coxeter arrangement for $W_{\langle s\rangle}$ in $V^*/\mathbb R\rho_s$;
  as it stands we can still
  identify the chambers of this subarrangement in $V^*$
  with the elements of $W_{\langle s\rangle}$; the difference compared to the usual Coxeter arrangement is that these chambers all contain the line spanned by $\rho_s$. Similarly, we can view the Cambrian cones for $W_{\langle s\rangle}$ in
  $V^*$ instead of $V^*/\mathbb R\rho_s$.

  Now observe that we can also  take the Coxeter fan
  and Cambrian cones for $W_{\langle s\rangle}$ which we have been considering
  in $V^*$ and intersect them with $H_s$. This defines a copy of the Coxeter
  fan and Cambrian cones, drawn in $H_s$. Essentially, we are exploiting the
  natural identification of $V^*/\mathbb R\rho_s$ with $H_s$.

Note that if $y$ is a $sc$-sortable element, {then it is also
  $c$-sortable, and $\ell(sy)>\ell(y)$. In this case,} we have $\Cone_c(y)\subset H_s^+.$ 
Then we have the following property (see also \cite[Proposition 9.6]{RS2}, \cite[Lemma 6.2]{RS1}, \cite[Lemma 3.12]{HLT}). 

\begin{prop}\label{cam generate}
Let $y\in W_{\langle s\rangle}$ be $sc$-sortable. 
Then $\Cone_c(y)$ is generated by $\Cone_{sc}(y)$
drawn in $H_s$ and by $\rho_s$.
\end{prop}

\begin{proof} By definition, the bounding hyperplanes of $\Cone_c(y)$ are
  $H_s$ and the bounding hyperplanes of $\Cone_{sc}(y)$.  These latter
  hyperplanes are reflection hyperplanes from $W_{\langle s\rangle}$,
  and therefore all pass through $\rho_{ s}$. \end{proof}

Example \ref{chamber} also gives an example of this proposition.

  We write $\Inv(w)$ for $w (\Phi^-) \cap \Phi^+$. This is a finite set of
  positive roots, and it characterizes $w$: the roots in
  $\Inv(w)$ are the labels of the hyperplanes which separate $wD$ from $D$. 

In Section \ref{back-cg}, we have already discussed the map from
  $W$ to $W_{\langle s\rangle}$ sending $w$ to $w_{\langle s\rangle}$, the longest
  prefix of $w$ contained in $W_{\langle s\rangle}$. This map can also be described
  in terms of $\Inv$: the element
  $w_{\langle s\rangle}$ can be characterized by the property that
  $\Inv(w_{\langle s\rangle})=\Inv(w)\cap \Phi_{\langle s\rangle}$.
Geometrically, we can think of $W_{\langle s\rangle}$ as the reflection group
  generated by reflections in the hyperplanes corresponding to roots in
  $\Phi_{\langle s\rangle}$.  Then $w_{\langle s\rangle}$ is the element of  
  $W_{\langle s\rangle}$ corresponding
to the chamber for $W_{\langle s\rangle}$ which contains $wD$.

  From the previous proposition, we get the following:

  \begin{lemm}\label{sc-fibre} Let $y$ be an $sc$-sortable element of
    $W_{\langle s\rangle}$. The $\pi^c$ fibre over $y$ consists of the elements
    of $W$ such that $w_{\langle s\rangle}$ is in the $\pi^{sc}$ fibre of
    $y$ and such that $\alpha_s$ is not in the inversion set of $w$.
  \end{lemm}

  \begin{proof} The first condition is exactly the condition of being on the
    {correct} side of the bounding hyperplanes of $\Cone_{sc}(y)$, while the second
    condition is equivalent to being on the {correct} side of $H_s$.
    \end{proof}

  



    We have the following important theorem from \cite{RS2}:

  \begin{thm}[{ \cite[Theorem 7.3]{RS2}}]\label{rs-th} For $X$ any subset of $W$, \begin{itemize}
\item[(a)] if $X$ is non-empty then
      $\bigwedge_{w\in {X}}\pi^c(w)= \pi^c(\bigwedge_{w\in {X}} w)$, and
\item[(b)] if $X$ has an upper bound then
$\bigvee_{w\in {X}}\pi^c(w)= \pi^c(\bigvee_{w\in {X}} w)$
    \end{itemize}
  \end{thm}

        The following useful lemma is an immediate consequence:

        \begin{lemm} \label{fibre-lemma} Let $X$ be a collection of elements in a single
          fibre for
  $\pi^c$. Their meet is also in the same fibre. Their join, if it exists,
          is also in the same fibre. \end{lemm}

        The following technical lemma plays a key role in our arguments.
        
        \begin{lemm} \label{infinite-lemma} Suppose that $r$ and $t$ are simple reflections which
          generate an infinite dihedral group, and suppose that we have
          three elements of $W$, say $z$, $zr$, and $zt$, with
          $\ell(zr)=\ell(zt)=\ell(z)+1$. Suppose $zr$ is in the same
          $\pi^c$ fibre as $z$. Then all the elements
          of the form $zrtrt\dots$ are contained in the same fibre as $z$,
          and all the elements of the form $ztrtr\dots$ are in the same
          fibre as $zt$.
        \end{lemm}

        (Note that the statement of the lemma covers both the case where $zt$ and $z$ are in different fibres, and where they are in the same fibre.)
        

        \begin{proof}
          Suppose that $z$ is in the fibre over the $c$-sortable
          element $y$. We will prove the result by induction on length of $y$
          and
          rank of $W$.

Suppose first that $\ell(y)>\ell(sy)$.  
Left-multiplication by $s$ is an order-preserving bijection from 
$A=\{w\in W\ | \ell(sw)<\ell(w)\}$ to $B=\{w\in W\ | \ell(sw)>\ell(w)\}$.
By the recursive definition of $\pi^c(w)$ ({see Definition \ref{map pi}}),
left multiplication by $s$ sends the fibres of $\pi^c$ in
$A$ to fibres of $\pi^{scs}$ in $B$. 
By induction on length, the statement holds for $sy$ and therefore for $y$. 

Suppose then that {$\ell(y)<\ell(sy)$}, so $y$ is an $sc$-sortable element of
$W_{\langle s \rangle}$. 
          The hyperplane between $z D$ and $zr D$ is perpendicular to a
certain positive root $\beta$, and the hyperplane between $z$ and $zt$
is perpendicular to a positive root $\gamma$. Suppose first that $\beta$ and $\gamma$ both lie in
$\Phi_{{\langle s\rangle}}$. 
Since
$\beta$ and $\gamma$ are in $\Phi_{{\langle s\rangle}}$, the corresponding
  hyperplanes also bound the chamber for $W_{\langle s\rangle}$ corresponding to $z_{\langle s \rangle}$. By the induction hypothesis
  for $W_{\langle s \rangle}$, the $W_{\langle s\rangle}$-chambers corresponding to
  $zrt$, $zrtr$, \dots all lie in $\Cone_{sc}(y)$, and
  the $W_{\langle s\rangle}$-chambers corresponding to
  $zt$, $ztr$, \dots all lie in $\Cone_{sc}(\pi(zt))$. 
  The $W$-chambers corresponding to all
  these elements also lie on the correct side
  of $H_s$ (since $zD$ does, and none of the hyperplanes crossed in passing
  from $z$ to $zr$ to $zrt$ to $zrtr$, etc., is $H_s$). The desired result
  then follows from Lemma \ref{cam generate}.

Now suppose
at least one of $\beta$ and $\gamma$ does not
lie in $\Phi_{\langle s \rangle}$. The infinite dihedral group generated by $r$ and $t$ corresponds bijectively to the $W$-chambers
around the intersection of the hyperplanes perpendicular to
$\beta$ and $\gamma$. The roots corresponding to the
other hyperplanes which contain this intersection
are positive combinations of $\beta$ and $\gamma$, and therefore none
of them is in $\Phi_{{\langle s\rangle}}$. Also, none of them is $H_s$. The desired
result now 
follows from Lemma \ref{sc-fibre}.
\end{proof}


\begin{proof}[Proof of Lemma \ref{lemmatwo}]
  If $\widehat y^c$ exists, then all the elements $w$ satisfying
  $\pi^c(w)=y$ are of length at most $\ell(\widehat y^c)$, and therefore form a finite set.  
  



{Conversely, suppose that $\{w \in W\mid \pi^c(w)=y\}$ is finite.}
If the elements of the $\pi^c$ fibre over $y$ have a common upper
bound, then they have a join and by Lemma \ref{fibre-lemma}
the join is in the fibre, so the fibre has a maximum element.

We must now justify that the elements of the
$\pi^c$ fibre over $y$ have a common upper bound.
Suppose otherwise. If we number the elements
of the fibre $f_1,\dots, f_r$, and successively consider $f_1\vee f_2$,
$(f_1\vee f_2) \vee f_3$, etc., we will eventually find a pair of elements
of the fibre with no join. Rename these elements as $p$ and $q$.

We will now argue that we can find a pair of elements $p'$, $q'$ in the fibre which have no upper bound and either: \begin{itemize}
\item $p'$ and $q'$ cover the same element, or
\item $p'\wedge q'> p\wedge q$. \end{itemize}
Let $p\wedge q=z$. Choose $a,b$ so $p\geq a \gtrdot z$ and $q \geq b \gtrdot z$.
If $a$ and $b$ have no join, then we have found our pair of elements (satisfying the first condition). If they
do have a join, let it be $v$. If $p$ and $v$ have no upper bound, we have
found our pair of elements (satisfying the second condition). Othewise, let
$w=p\vee v$, and we can take $w$ and $q$ as our pair of elements (again
satisfying the second condition).

Since by assumption the fibre is finite, if we keep repeating this procedure,
we must eventually find a pair satisfying the first condition, so suppose
we have elements $p,q$ both covering $z$. Write $p=zt$, $q=zr$, with $t$ and
$r$ simple reflections. Lemma \ref{infinite-lemma} now shows that there
are infinitely many elements in the fibre over $y$, contradicting our assumption.

We have reached a contradiction starting from our assumption that the fibre
of $y$ does not have an upper bound, so it must be that the
fibre of $y$ does have an upper bound, and, as we have showed, it
therefore has a join.
\end{proof}

  \begin{proof}[Proof of Proposition \ref{inductive}]
    Suppose that there is an element of the fibre of $x$ which is not
    below $\widehat y^c$. Among such elements, choose one, $z$, so that
    $z\wedge \widehat y^c$ is maximal in weak order, and among
    $z$ which
    accomplish this, choose $z$ minimal in weak order.
   Let $m=z\wedge \widehat y^c$. 
    By Theorem \ref{rs-th},
    $m$ is also in the $\pi^c$ fibre of $x$.
     Since $z>m$,
      there exists $z'$ such that $z\geq z'\gtrdot m$. Since
      $m$ is the meet of $z$ and $\widehat y^c$, and $z'\leq z$, we cannot
      have $z'\leq \widehat y^c$. Thus if $z>z'$ we would have preferred to
      choose $z'$ instead of $z$. We conclude that $z=z'$, so
      $z \gtrdot m$.
    
    Now choose $p$
    a cover of $m$ which lies below
    $\widehat y^c$. If $p$ and $z$ had a join, it would be either in the
    fibre of $x$ or of $y$. In the former case, we should have chosen that
    join rather than $z$. But the latter case violates the hypothesis that
    $\widehat y^{c}$ is the maximum element of its fibre. So they have no
    join,
    and thus there is a copy of the infinite dihedral group starting
    at $m$ and going up through $z$ and $p$. This gives us an infinite chain
    above $p$. By Lemma \ref{infinite-lemma},
    all the elements of this dihedral group are in the fibres of $x$ or $y$.
    Consider $q$, the largest element on the chain up from $p$ which is
    below $\widehat y^{c}$. If $q$ is in the same fibre as $x$, then we should
    have chosen the next element up the chain from $q$ instead of $z$. If
    $q$ is in the same fibre as $y$, then the next element up the chain
    from $q$ is also in the same fibre as $y$, but this contradicts the
    hypothesis that $\widehat y^c$ is the maximum element of its fibre.

    We have reached a contradiction, so there must not be any element of the
    fibre of $x$ which is {not} below $\widehat y^c$. The elements of the fibre of
    $x$ therefore have an upper bound, so they have a join, and, by Lemma \ref{fibre-lemma}, that join is also in the fibre of $x$. So there is a
    top element of the fibre, $\widehat x^c$. Since all the elements of the
    fibre are below $\widehat y^c$, this is in particular true of $\widehat x^c$. \end{proof}

  \begin{proof}[Proof of Lemma \ref{max is max}] Suppose that $x$ is maximal
    in the fibre over $y$, but not maximum. It follows from
    Lemma \ref{lemmatwo} that the fibre is infinite. We know the join of
    two elements of the fibre, if it exists, is again in the fibre. So
    the join of $x$ with any element of the fibre not below $x$ does not exist.
    Find such a $p$ for which the meet with $x$ is as large as possible, and
    then choose {such a} $p$ as small as possible with that meet.
    This means we have an element
    $p$ with $p$ covering $x\wedge p$, and all the elements {$z$} satisfying
    $x\wedge p<z\leq x$ do not admit any covers which are not $\leq x$.  Let $r$ be a
    cover of $x\wedge p$ {such that $r\leq x$}. The join of $r$ and $p$ must not exist. Therefore
    there is a copy of the infinite dihedral group starting at
    $p\wedge r$, and by Lemma \ref{infinite-lemma}, all its elements are in
    the same $\pi^c$ fibre. Let $q$ be the largest element of the chain through $r$ which
    is below $x$, and let $v$ be the next element up the chain. If $q<x$, then
    we would have preferred $v$ to $p$. But if $q=x$, then $v>x$ contradicts
    that $x$ is maximal in the fibre. We have found a contradiction.
    \end{proof}


\section{Cofinite torsion classes and bounded $c$-sortable elements}

Next we study a relationship between cofinite torsion classes and $c$-sortable elements. 
First we give the following easy lemma.

\begin{lemm}\label{existence sortable}
Let $w\in W$.
If $\overline{\C I_w}$ is a torsion class, then it is functorially finite.
In particular, there exists an $\Ext$-projective $KQ$-module $T$ of $\overline{\C I_w}$ such that 
$\Fac T=\overline{\C I_w}$.
\end{lemm}

\begin{proof}
Since $\overline{\C I_w}$ is cofinite, the corresponding torsion free class is finite and, in particular, it is functorially finite. 
Hence so is $\overline{\C I_w}$ by \cite{S} (see also \cite[Proposition 1.1]{AIR}).
Thus, the statement follows from Theorem \ref{bij support tilt}.
\end{proof}

We recall the following result from \cite[Lemma 3.1]{ORT} (and the sentence before it).

\begin{lemm}\cite[Lemma 3.1]{ORT}\label{ORT relative} For  finite dimensional $\La$-modules $M$ and $N$, 
there is a surjective morphism 
$$\Hom_{KQ}(\tau^-(M_{KQ}),N_{KQ})\to\underline{\Ext}_\Lambda^1(M,N),$$
 where $\underline{\Ext}_\Lambda^1(-,-)$ denotes the subfunctor of ${\Ext}_\Lambda^1(-,-)$ given by $KQ$-split short exact sequences.
\end{lemm}

Then we have the following proposition.

\begin{prop}\label{C-facI}
Let $w\in W$.
If $\overline{\C I_w}$ is a torsion class, 
then we have 
$$\Fac I_w\cap\mod KQ=\overline{\C I_w}.$$
\end{prop}

\begin{proof}
Clearly we have $\Fac I_w\cap\mod KQ\subset\overline{\C I_w}.$
We will show the converse. 

By Lemma \ref{existence sortable}, 
there exists Ext-projective $T$ in $\overline{\C I_w}$ such that $\Fac T=\overline{\C I_w}$. 
Let $P$ be an indecomposable direct summand of $T$. 
It is enough to show that $P$ is contained in $\Fac I_w\cap\mod KQ$.  

Since $\overline{\C I_w}$ is cofinite and contains all non-preprojective $KQ$-modules, 
$P$ is preprojective and hence $P$ is in $\C I_w$. 
Therefore, by taking large $N$, $P$ is a subquotient of the finite dimensional module $I_w/I_{c^N}$. 
It is therefore a submodule of a finite dimensional
quotient of $I_w$. Hence we have a $KQ$-split exact sequence of finite dimensional modules 
$$0\rightarrow P \rightarrow E \rightarrow M\rightarrow 0$$
with $E,M$ in $\Fac I_w$.

By Lemma \ref{ORT relative}, 
$\underline{\Ext}_\Lambda^1(M,P)$ is a quotient of $\Hom_{KQ}(\tau^-(M_{KQ}),P)$, 
which is isomorphic to $D\Ext^1_{KQ}(P,M_{KQ})$. 
Since $M_{KQ}$ is in ${\C I_w}$, we have $\Ext^1_{KQ}(P,M_{KQ})=0$ 
because $P$ is Ext-projective in $\overline{\C I_w}$.  
Thus, $P$ is itself a quotient of $I_w$ and hence $P$ is in $\Fac I_w\cap\mod KQ$. 
\end{proof}

As a consequence, we have the following result, which shows \cite[Conjecture 11.2]{ORT}.

\begin{thm}\label{torsion=pi}
Let $w\in W$.
If $\overline{\C I_w}$ is a torsion class, then the corresponding torsion 
free class $(\overline{\C I_w})^{\perp}$ is given by $\C \La_{\pi^c(w)}$.
\end{thm}

\begin{proof}
Corollary \ref{torsionpath} implies that $(\Fac I_w\cap\mod KQ,\Sub\La_w\cap\mod KQ)$ is a torsion pair of $\mod KQ$. 
On the other hand, by Proposition \ref{C-facI}, we have $\Fac I_w\cap\mod KQ=\overline{\C I_w}$. 
Thus Theorem \ref{main1} (b) shows the assertion.
\end{proof}

Next we give the following definition.

\begin{defi}\label{def well}
{If $Q$ is a non-Dynkin quiver, then} 
a $c$-sortable element $x$ is called \emph{bounded $c$-sortable}
if there exists a positive integer $N$ such that $x\leq c^N$.
{If $Q$ is a Dynkin quiver, then} we regard any $c$-sortable element as bounded $c$-sortable. 
We denote by $\wcsort W$ the set of bounded $c$-sortable elements. 
\end{defi}

\begin{exam}\label{exam bounded}
\begin{itemize}
\item[(a)] Let $Q$ be the following quiver 
$$\xymatrix@C10pt@R5pt{
1\ar[r]&2\ar@{=>}[r]&3.}$$
Because 
\begin{eqnarray*}
c^3&=&s_1s_2s_3s_1s_2s_3s_1s_2s_3\\
&=&s_1s_2s_3s_1s_2s_1s_3s_2s_3\\
&=&s_1s_2s_3s_2s_1s_2s_3s_2s_3,
\end{eqnarray*}
we have $s_1s_2s_3s_2\leq c^3$ and hence $s_1s_2s_3s_2$ is bounded $c$-sortable.
\item[(b)]Let $Q$ be the following quiver 
$$\xymatrix@C10pt@R5pt{&2\ar[rd]&\\
1\ar[ru]\ar[rr]&&3.}$$
Then one can check that $s_1s_2s_3s_2$ is not bounded $c$-sortable. 
\end{itemize}
\end{exam}

Then we give the following lemma.

\begin{lemm}\label{c-bounded}
Let $x$ be a $c$-sortable element. Then the following are equivalent.
\begin{itemize}
\item[(i)]$x$ is bounded $c$-sortable. 
\item[(ii)]Any module of $\C\La_x$ is a preprojective module.
\item[(iii)]The corresponding torsion class ${}^{\perp}(\C \La_x)$ is cofinite.
\end{itemize}
\end{lemm}

\begin{proof}
It is enough to consider the non-Dynkin case. 

Then we have $\C \La_{c^N}=\add\{KQ,\tau^{-1}(KQ),\ldots,\tau^{-N}(KQ)\}$ by Theorem \ref{equality} (c). 
On the other hand, by Proposition \ref{pathorders}, we have $x\leq c^N$ if and only if $\C \La_x\subset\C\La_{c^N}$. 
Thus it implies the the equivalence of (i) and (ii).
The equivalence of (ii) and (iii) is straightforward from the structure of the AR quiver.
\end{proof}

Moreover, it is convenient to introduce the following terminology.

\begin{defi}\label{def hat}
Let $x$ be a $c$-sortable element. If there exists a maximum element amongst $w\in W$ satisfying $\pi^c(w)=x$, then we denote it by $\widehat{x}^c=\widehat{x}$ and call it \emph{$c$-antisortable}. We denote by $\comp W$ the set of $c$-antisortable elements of $W$.
\end{defi}


\begin{exam}\label{exam hat}
\begin{itemize}
\item[(a)]
Let $Q$ be the following quiver 
$$\xymatrix@C10pt@R5pt{
1\ar[r]&2\ar@{=>}[r]&3.}$$
Take a $c$-sortable element $x=s_1s_2s_3s_2$. 
Then one can check that  $\widehat{x}=s_1s_2s_3s_2s_1$.
\item[(b)]
Let $Q$ be the following quiver 
$$\xymatrix@C10pt@R5pt{&2\ar[rd]&\\
1\ar[ru]\ar[rr]&&3.}$$
Take a $c$-sortable element $x=s_1s_2s_3s_2$.
Consider the following infinite word 
$$s_1s_2s_3s_2s_1s_3s_2s_1s_3s_2s_1s_3s_2\cdots.$$ 
Then from the word, any arbitrarily long prefix  $w$ will satisfy 
$\pi^c(w)=x$. 
Thus, $\widehat{x}$ does not exist.
\end{itemize}
\end{exam}

We call a torsion pair \emph{cofinite} if the torsion class is cofinite.
Then we give a description of $w$ when $\overline{\C I_w}$ is a torsion class as follows.

\begin{thm}\label{bound property}
\begin{itemize}
\item[(a)]Let $w\in W$.  
If $\overline{\C I_w}$ is a torsion class, 
then $\pi^c(w)$ is bounded $c$-sortable and $w=\widehat{\pi^c(w)}$.
\item[(b)]Let $x$ be a $c$-sortable element. 
If $x$ is bounded $c$-sortable, then there exists $\widehat{x}$ and $\overline{\C I_{\widehat{x}}}$ is a torsion class. 
\item[(c)] 
There is a bijection $$\wcsort W \longrightarrow \{\textnormal{cofinite torsion pairs of }\mod KQ\},\ \ \ x\mapsto (\overline{\C I_{\widehat{x}}},\C \La_x).$$
\end{itemize}
\end{thm}

\begin{proof}
(a) Assume that  $\overline{\C I_w}$ is a torsion class. 
By Theorem \ref{torsion=pi}, 
we have $\overline{\C I_w}={}^{\perp}(\C \La_{\pi^c(w)})$. 
Then, Lemma \ref{c-bounded} implies that $\pi^c(w)$ is bounded $c$-sortable.

Assume that there exists $u\in W$ with $\ell(u)>\ell(w)$ and
$\pi^c(u)=\pi^c(w)$.  Then by Theorem \ref{main1}, we have 
$$\overline{\C(I_{u})}\supset\Fac I_{u}\cap\mod KQ= {}^{\perp}(\C \La_{\pi^c(w)})=\overline{\C I_w}.$$ 
Because $\overline{\C(I_{u})}$ (respectively, $\overline{\C I_w}$) consists of the category which is obtained by removing $\ell(u)$ (respectively, $\ell(w)$) indecomposable $KQ$-modules from the preprojective modules by Theorem \ref{ORT-thm}, this is impossible. 

Thus the set of $y\in W$ with $\pi^c(y)=\pi^c(w)$ consists of element of length at most $\ell(w)$, which is a finite set.  By Lemma \ref{lemmatwo}, there exists
a maximum element of this fibre.  Since $w$ is of maximal length
in the fibre,
this maximum element must be $w$ by Lemma \ref{max is max} .

(b) By Lemma \ref{c-bounded}, 
${}^{\perp}(\C \La_x)$ is a cofinite torsion class. 
Hence we have ${}^{\perp}(\C \La_x)=\overline{\C I_w}$ for some $w\in W$ from Theorem \ref{ORT-thm}. 
Then Theorem \ref{torsion=pi} implies $\pi^c(w)=x$ and (a) shows
that $w=\widehat{x}$. 

(c) This follows from the above (a), (b), Corollary \ref{sortable bij} and \ref{torsion=pi} and Lemma \ref{c-bounded}.
\end{proof}

Thus, these torsion pairs can be controlled by bounded $c$-sortable elements and $c$-antisortable elements. 

\begin{exam}
\begin{itemize}
\item[(a)]
Let $Q=(1\to 2\to 3)$. 
Then the AR quiver of $\mod KQ$ is given by 

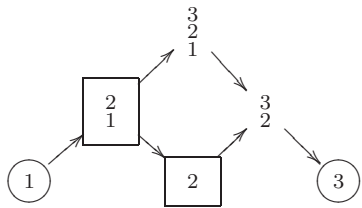
\[
\xymatrix@C10pt@R5pt{
 & & {\begin{smallmatrix}3\\2\\ 1\end{smallmatrix}} \ar[rd]& & \\
 & *++[F]{{\begin{smallmatrix}2\\ 1\end{smallmatrix}}}\ar[ru]  \ar[rd]& &  {\begin{smallmatrix}3\\2\end{smallmatrix}} \ar[rd] &\\
*++[o][F]{{\begin{smallmatrix}1\end{smallmatrix}}} \ar[ru]& &  *++[F]{{\begin{smallmatrix}2\end{smallmatrix}}}  \ar[ru] & &  *++[o][F]{{\begin{smallmatrix}3\end{smallmatrix}}} .}
\]

For example, we take a $c$-sortable element $x=s_1s_3$. 
Then we have the torsion free class $\C \La_x=\add\{{\begin{smallmatrix}1\end{smallmatrix}},{\begin{smallmatrix}3\end{smallmatrix}}\}$ whose modules are circled above.
The corresponding torsion class is $\add\{{\begin{smallmatrix}2\\ 1\end{smallmatrix}},{\begin{smallmatrix}2\end{smallmatrix}}\}$ whose modules are squared above.
By Theorem \ref{ORT-thm}, it is given as $\C I_w$ for $w=s_1s_3s_2s_1$. 
Then one can check that this elements gives $\widehat{x}$. 


\item[(b)]
Let $Q=(1\to 2\Rightarrow 3)$. 
Then the preprojective component of the AR quiver of $\mod KQ$ is given a translation quiver \cite{ASS} of the following form:

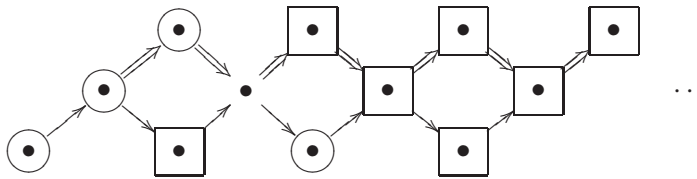
\[
\xymatrix@C10pt@R5pt{
 & & *++[o][F]{\bullet} \ar@{=>}[rd]& &*++[F]{\bullet}\ar@{=>}[rd] &&*++[F]{\bullet}\ar@{=>}[rd]&&*++[F]{\bullet}&\\
 & *++[o][F]{\bullet} \ar@{=>}[ru]   \ar[rd]& & \bullet\ar[rd]  \ar@{=>}[ru]  &&*++[F]{\bullet} \ar@{=>}[ru]  \ar[rd] &&*++[F]{\bullet} \ar@{=>}[ru]   &&\cdots\\
*++[o][F]{\bullet}  \ar[ru]& & *++[F]{\bullet}  \ar[ru] & & *++[o][F]{\bullet} \ar[ru]&&*++[F]{\bullet}\ar[ru]&&&}
\]

For example, we take a $c$-sortable element $x=s_1s_2s_3s_2$, which is bounded $c$-sortable.
Then $\C \La_x$ consists of the modules which are circled above.
The corresponding torsion class consists of the modules which are squared above and all the rest.
It is given as $\overline{\C I_w}$ for $w=s_1s_2s_3s_2s_1$.
Therefore our theorem implies that we have $\widehat{x}=s_1s_2s_3s_2s_1$. 
\end{itemize}
\end{exam}


\section{Maximum elements of $\pi^c$-fibres}

In the previous section, we showed that there exists $\widehat{x}^c$ for a $c$-sortable element $x$ if $x$ is bounded $c$-sortable. 
Our final aim is to show the converse.

\begin{thm}\label{general case} 
Let $x$ be a $c$-sortable element. Then the following are equivalent.
\begin{itemize}
\item[(i)]$x$ is bounded $c$-sortable.
\item[(ii)]There exists $\widehat{x}$.
\end{itemize} 
In particular, there is a bijection 
$$\xymatrix@C60pt@R30pt{
\wcsort W \ar@<0.5ex>[r]^{\widehat{(-)}}   &\comp W.\ar@<0.5ex>[l]^{\pi^c(-)}  }$$
\end{thm}

In the Dynkin case, the statement follows from \cite{R3}, so that we assume that $Q$ is non-Dynkin in this section.

Before proving the theorem, we first deal with a special class of $c$-sortable elements.

\begin{prop}\label{special case}
Let $x$ be a $c$-sortable element. 
Assume that $x$ is unbounded, but any $c$-sortable element $x'$ with $x'<x$ is bounded.
Then we have $\{w\in W\ |\ \pi^c(w)=x\}$ is infinite.
\end{prop}

\begin{remk}
The above condition on $x$ is equivalent to saying that 
any proper torsion free subcategory of $\C\La_x$ is in the preprojective component by Lemma \ref{c-bounded}. 
Equivalently any torsion class properly containing ${}^{\perp}(\C \La_x)$ is cofinite.
\end{remk} 

From now on, we fix the above $x$, that is, $x$ is an unbounded $c$-sortable element, such that any $c$-sortable element $x'$ with $x'<x$ is bounded.
Then we give the following observation. 

\begin{lemm}\label{preproj ext}
In the above setting, there exists an $\Ext$-projective $T$ of ${}^{\perp}(\C \La_x)$ such that $T$ is preprojective and 
$$\Fac T={}^{\perp}(\C \La_x).$$
\end{lemm}

\begin{proof}
  Let $\xx=s_{u_1}\cdots s_{u_l}$ be a $c$-sortable reduced expression of $x$,
  and let $\xx'=s_{u_1}\cdots s_{u_{l-1}}$.  Let $T'$ be the sum of the $\Ext$-projective indecomposable modules of ${}^\perp (\C \Lambda_{x'})$.  We know that
  $T'$ is preprojective and $\Fac T'={}^\perp (\C \Lambda_{x'})$ by our assumption. In the poset of
  torsion free classes, $\Sub\Lambda_{x}$ covers $\Sub\Lambda_{x'}$, so
  ${}^\perp (\C \Lambda_{x'})$ covers ${}^\perp (\C \Lambda_{x})$ in the poset of
  torsion classes.  It follows from \cite{AIR} that the Ext-projectives of
  ${}^\perp (\C \Lambda_{x})$
  can be obtained from $T'$ by a single mutation; that
  is to say, there is an indecomposable summand $T'_1$ of $T'$ such that
  the summands of $T'/T'_1$ are also Ext-projectives of ${}^\perp (\C \Lambda_{x})$,
  and the remaining indecomposable Ext-projective of ${}^\perp (\C \Lambda_{x})$
  is either 0 or is the righthand term in the following exact sequence:
  $$ T'_1 \rightarrow U \rightarrow T_1 \rightarrow 0,$$
  where $U$ is a minimal left $\add(T'/T'_1)$-approximation of $T'_1$.  Since $T_1$ is a
  quotient of $U\in \Fac (T'/T'_1)$, we know that $\Fac(T'/T'_1)=\Fac(T'/T'_1 \oplus T_1)= {}^\perp (\C \Lambda_{x})$,
  so $T'/T'_1$ satisfies the hypotheses of the lemma.
  \end{proof}





Next we recall a result of \cite{ORT}.
Fix the infinite word  $\cc^\infty=\cc \cc
\cc\ldots,$ where $\cc=s_1\ldots s_n$. 
We say that an infinite subword $\ww$ of $\cc^\infty$ is {\it leftmost} if,
for all $m\geq 0$, the subword of $\cc^\infty$ consisting of the 
first $m$ letters of $\ww$ is leftmost. Then we have the following result. 

\begin{theorem}\cite[Theorem 8.1]{ORT}
There is a bijection between the leftmost subwords of $\cc^\infty$ and the quotient closed subcategories of the preprojective component $\PP$. The word corresponds to the missing indecomposable modules of the subcategory.
\end{theorem}

In our case, because $x$ is unbounded $c$-sortable, ${}^{\perp}(\C \La_x)$ is not cofinite by Lemma \ref{c-bounded}. 
Then, from the above theorem, there exists the infinite leftmost word 
$$s_{u_1}\dots s_{u_j}\dots$$ 
which corresponds to the indecomposable modules which are not in ${}^{\perp}(\C \La_x)\cap\PP$. 

Under the above setting, let $v_j:=s_{u_1}\dots s_{u_j}$ for $l\geq 0$. Then we have ${}^{\perp}(\C \La_x)\cap\PP=\bigcap_{j\geq0} \C I_{v_j}$. 
Note that 
we have $j\leq i$ if and only if $v_j\leq v_i$, or equivalently, $\Fac I_{v_j}\supset \Fac I_{v_i}$ by Proposition \ref{orders}. 
Our aim is to show that we can take arbitrary large $j$ such that $\pi^c(v_j)=x$.

By modifying Proposition \ref{C-facI}, we obtain the next statement.

\begin{lemm}\label{bounded}
Under the above setting, for any $v_j$, we have 
\[\Fac I_{v_j}\cap\mod KQ \supset {}^{\perp}(\C \La_x).\]
\end{lemm}

\begin{proof}
By Lemma \ref{preproj ext}, it is enough to show that 
we have $P\in\Fac I_{v_j}\cap\mod KQ$ for any indecomposable preprojective Ext-projective $P$ of ${}^{\perp}(\C \La_x)$. 
Since ${}^{\perp}(\C \La_x)\cap\PP=\bigcap_{j\geq0} \C I_{v_j}$, we have $P\in\C I_{v_j}$ for all $j$.  Choose $j$ sufficiently large that $\Ext^1_{KQ}(P,M)=0$ for
all $M$ in $\C I_{v_{j}}$.  
Thus, as in the proof of Proposition \ref{C-facI}, we have a $KQ$-split short exact sequence of 
finite dimensional $\Lambda$-modules 
$$0\rightarrow P \rightarrow E \rightarrow M\rightarrow 0$$
with $E,M$ in $\Fac I_{v_j}$.

Since $M_{KQ}\in\C I_{v_j}$, we have $\Ext^1_{KQ}(P,M_{KQ})=0$. 
Because Lemma \ref{ORT relative} implies that $\underline{\Ext}_\Lambda^1(M,P)$ is a quotient of $\Hom_{KQ}(\tau^-(M_{KQ}),P)\simeq D\Ext^1_{KQ}(P,M_{KQ})$, 
we have $P\in\Fac I_{v_j}\cap\mod KQ$.  This holds for all $j$ sufficiently
large, and since 
$\Fac I_{v_j}\cap\mod KQ$ is monotonically decreasing, it holds for all $j$.
Thus $\Fac I_{v_j}\cap\mod KQ\supset {}^{\perp}(\C \La_x)$ for all $j$. 
\end{proof}



Moreover we have the following result. 

\begin{lemm}\label{no bounded}
Under the above setting,  there exists $k$ such that for all $l\geq k$ 
\[\Fac I_{v_l} \cap\mod KQ ={}^{\perp}(\C \La_x).\] 
\end{lemm}

\begin{proof}
Assume that 
$\Fac I_{v_j} \cap\mod KQ \supsetneq{}^{\perp}(\C \La_x)$ for any $j$. Then our assumption shows the torsion class $\Fac I_{v_j} \cap\mod KQ$ is cofinite. 
 Since ${}^{\perp}(\C \La_{\pi^c(v_j)})=\Fac I_{v_j} \cap\mod KQ \supsetneq{}^{\perp}(\C \La_x)$, we have $\C \La_{\pi^c(v_j)}\subsetneq\C \La_x$. Because $\C \La_x$ is finite, there are only finitely many distinct torsion classes in $\{\Fac I_{v_j} \cap\mod KQ\ |\ j\geq0\}$ and hence $\bigcap_{j\geq0}(\Fac I_{v_j}\cap\mod KQ\cap\PP)$ is also cofinite. 
 
On the other hand, since $\Fac I_{v_j} \cap\mod KQ\cap\PP\subset\C I_{v_j}$, 
we have 
$$\bigcap_{j\geq0}(\Fac I_{v_j}\cap\mod KQ\cap\PP)\subset \bigcap_{j\geq0} \C I_{v_j}={}^{\perp}(\C \La_x)\cap\PP,$$
which is a contradiction because the right hand side is not cofinite. 

Therefore, by Lemma \ref{bounded}, there exists $k$ such that $\Fac I_{v_l}\cap\mod KQ={}^{\perp}(\C \La_x)$ for all $l\geq k$.
\end{proof}

Then we give a proof of Proposition \ref{special case}.

\begin{proof}[Proof of Proposition \ref{special case}]
By Theorem \ref{main1} and Lemma \ref{no bounded}, 
there exists $k$ such that, for all $l\geq k$, 
\[{}^{\perp}(\C \La_{\pi^c(v_l)})=\Fac I_{v_l}\cap\mod KQ ={}^{\perp}(\C \La_{x}).\] 
Thus Proposition \ref{pathorders} shows that $\pi^c(v_l)=x$ for all $l\geq k$ and the conclusion follows.
\end{proof}

Finally, we obtain a proof of Theorem \ref{general case}.

\begin{proof}[Proof of Theorem \ref{general case}]
By Theorem \ref{bound property}, (i) implies (ii). 
We show that (ii) implies (i).

Let $x$ be a $c$-sortable element which is not bounded $c$-sortable. 
Take an unrefinable chain of $c$-sortable elements $x_1<\cdots<x_l=x$.
Then there exists $x_{k}$ for some $1\leq k \leq l$ satisfying the condition of Proposition \ref{special case}. 
Therefore we have $\{w\in W\ |\ \pi^c(w)=x_{k}\}$ is infinite.  
Then, Proposition \ref{inductive} shows the claim.  
The second statement follows from the first statement. 
\end{proof}

Consequently, we give a proof of the following conjecture, which was shown in the Dynkin case \cite[Proposition 11.4]{ORT}.

\begin{cor}\cite[Conjecture 11.1]{ORT}\label{torsion conj}
The following conditions are equivalent. 
\begin{itemize}
\item[(i)] $\overline{\C I_w}$ 
is a torsion class.
\item[(ii)] For every $i\in Q_0$ such
that $\ell(ws_i)>\ell(w)$, we have that $\pi^c(ws_i)>\pi^c(w)$.\end{itemize}
\end{cor}

\begin{proof}
(i) implies (ii). This follows from Theorem \ref{bound property} (a). 

  (ii) implies (i). Condition (ii) amounts to saying that $w$ is maximal
  among $\{x \in W \mid \pi^c(x)=\pi^c(w)\}$. Lemma \ref{max is max} tells
  us that $w$ is actually the maximum element, i.e., 
  $w=\widehat{\pi^c(w)}$.  By Theorem \ref{general case}, $w$ is
  bounded $c$-sortable.  
Then Theorem \ref{bound property} (b) implies that $\overline{\C I_w}$ is a torsion class.
\end{proof}


\end{document}